\documentclass[12pt]{amsart}
\usepackage{amssymb,amsmath,amsthm,hyperref,mathrsfs,multirow}
\newtheorem{theorem}{Theorem}[section]
\newtheorem{lemma}[theorem]{Lemma}
\newtheorem{cor}[theorem]{Corollary}

\newtheorem{prop}[theorem]{Proposition}

\theoremstyle{definition}

\theoremstyle{remark}
\newtheorem{remark}[theorem]{Remark}

\numberwithin{equation}{section}

\newcommand{\abs}[1]{\lvert#1\rvert}
\newcommand{\spt}{\mbox{\rm spt}}

\newcommand{\qbin}[2]{\left[\begin{matrix} #1 \\ #2 \end{matrix} \right]}
\newcommand{\bin}[1]{\binom{#1}{2}}
\newcommand{\qbas}[5]{{}_2\phi_1\left(\begin{matrix} 
    #1, &#2; &#3, &#4 \\ 
        &#5  &    & 
\end{matrix}\right)}
\newcommand{\qbasc}[7]{{}_3\phi_2\left(\begin{matrix} 
    #1, &#2, &#3; &#4, &#5 \\ 
        &#6, &#7    & 
\end{matrix}\right)}
\newcommand{\LHS}{\mbox{LHS}}
\newcommand{\RHS}{\mbox{RHS}}
\newcommand{\sgn}{\textnormal{sgn}}
\newcommand{\Parans}[1]{\left(#1\right)}
\newcommand{\SB}{\overline{\mbox{\rm S}}}
\newcommand{\sptBar}[2]{\overline{\mbox{\normalfont spt}}_{#1}\Parans{#2}}
\newcommand{\STwoB}{\mbox{\rm S2}}
\newcommand{\Mspt}[1]{\mbox{\normalfont M2spt}\Parans{#1}}
\newcommand\leg[2]{\genfrac{(}{)}{}{}{#1}{#2}} 
\allowdisplaybreaks

\begin{document}
\newcommand\mylabel[1]{\label{#1}}
\newcommand{\beqs}{\begin{equation*}}
\newcommand{\eeqs}{\end{equation*}}
\newcommand{\beq}{\begin{equation}}
\newcommand{\eeq}{\end{equation}}
\newcommand\eqn[1]{(\ref{eq:#1})}
\newcommand\thm[1]{\ref{thm:#1}}
\newcommand\lem[1]{\ref{lem:#1}}
\newcommand\propo[1]{\ref{propo:#1}}
\newcommand\corol[1]{\ref{cor:#1}}
\newcommand\sect[1]{\ref{sec:#1}}

\title[Two-variable Hecke-Rogers identities]
{Universal Mock Theta Functions \\
and 
\\
Two-variable Hecke-Rogers identities}

\author{F. G. Garvan}
\address{Department of Mathematics, University of Florida, Gainesville,
FL 32611-8105}
\email{fgarvan@ufl.edu}
\thanks{A preliminary version of this work was given earlier in
the Experimental Mathematics Seminar, Rutgers University, April 25, 2013.
See 
\url{http://youtu.be/oz2mdkd5jX4}
for the online video.}


\date{\today}                   



\begin{abstract}
We obtain two-variable Hecke-Rogers identities for three universal mock theta
functions. This implies that many of Ramanujan's mock theta functions,
including all the third order functions, have a Hecke-Rogers-type
double sum representation. We find new generating function identities for
the Dyson rank function, the overpartition rank function, the
$M2$-rank function and related spt-crank functions. Results are proved using
the theory of basic hypergeometric functions.
\end{abstract}

\maketitle

\section{Introduction}
\mylabel{sec:intro}

In this paper we obtain two-variable generalizations of the
following Hecke-Rogers identities.

\begin{align}
\prod_{n=1}^\infty (1 - q^n)^2
&=
\sum_{n=0}^\infty \sum_{m=-[n/2]}^{[n/2]} (-1)^{n+m} 
q^{\frac{1}{2}(n^2-3m^2) + \frac{1}{2}(n+m)},
\mylabel{eq:HR1}
\\
\prod_{n=1}^\infty (1 - q^n) (1 - q^{2n})
&=
\sum_{n=0}^\infty \sum_{m=-[n/2]}^{[n/2]} (-1)^{n+m} 
q^{\frac{1}{2}(n^2-2m^2) + \frac{1}{2}n},
\mylabel{eq:HR2}
\\
\prod_{n=1}^\infty (1 - q^n) (1 - q^{2n})
&=
\sum_{n=0}^\infty \sum_{m=-[n/3]}^{[n/3]} (-1)^{n} 
q^{\frac{1}{2}(n^2-8m^2) + \frac{1}{2}n},
\mylabel{eq:HR3}
\\
\prod_{n=1}^\infty (1 - q^n) (1 - q^{2n})
&=
\sum_{n=0}^\infty \sum_{m=-n}^{n} (-1)^{n} 
q^{\frac{1}{2}(2n^2-m^2) + \frac{1}{2}(2n +m)}.
\mylabel{eq:HR4}
\end{align}
Hecke \cite{He25} was the first to systematically consider identities 
of this type. Equation \eqn{HR1} was found by Hecke 
\cite[Equation (7),p.425]{He25} but is originally due to
L.~J.~Rogers \cite[p.323]{Ro1894}.
Identities of this type arose in Kac and Petersen's \cite{KP} work
on character formulas for infinite dimensional Lie algebras and string
functions. Equation \eqn{HR3} is due to Kac and Petersen 
\cite[final equation]{KP}. Andrews \cite{An-84} derived \eqn{HR2}, \eqn{HR3}
using his constant term method. Bressoud \cite{Br86} derived \eqn{HR2}, 
\eqn{HR4} using $q$-Hermite polynomials.

Our generalizations of \eqn{HR1}--\eqn{HR4} are in terms of
the universal mock theta functions
\begin{align}
R(z,q) &= \sum_{n=0}^\infty \frac{q^{n^2}}
                                {(zq;q)_n (z^{-1}q;q)_n},
\mylabel{eq:Rdef}
\\
H(z,q) &= \sum_{n=0}^\infty \frac{(-1;q)_n q^{\frac{1}{2}n(n+1)}}
                                {(zq;q)_n (z^{-1}q;q)_n},
\mylabel{eq:Hdef}
\\
K(z,q) &= \sum_{n=0}^\infty \frac{(-1)^n (q;q^2)_n q^{n^2}}
                                {(zq^2;q^2)_n (z^{-1}q^2;q^2)_n}.
\mylabel{eq:Kdef}
\end{align}
Here and throughout the paper we use the standard $q$-notation.
\begin{align*}
(a;q)_{\infty} &= \prod_{k=0}^\infty (1-aq^k),
\\
(a;q)_{n} &= \frac{(a;q)_{\infty}}{(aq^n;q)_{\infty}},
\\
(a_1,a_2,\dots,a_j;q)_{\infty} 
&= (a_1;q)_{\infty}(a_2;q)_{\infty}\dots(a_j;q)_{\infty},
\\
(a_1,a_2,\dots,a_j;q)_{n} 
&=
(a_1;q)_{n}(a_2;q)_{n}\dots(a_j;q)_{n}.
\end{align*}
The functions \eqn{Rdef}--\eqn{Kdef} are called \textit{universal mock theta}
functions because Hickerson \cite{Hi88a}, \cite{Hi88b} and
Gordon and McIntosh \cite{Go-McI} have shown that each of the classical mock theta
functions may be expressed as specializations of these functions up to
the addition of a modular form.

We note that the functions $R(z,q)$, $H(z,q)$ and $K(z,q)$ are generating
functions for various rank-type functions. Let
\begin{align*}
N(m,n) 
&= \mbox{the number of partitions of $n$ with Dyson rank $m$ (\cite{Dy44})},\\
\overline{N}(m,n)
&= \mbox{the number of overpartitions of $n$ with Dyson rank $m$ (\cite{Br-Lo-Os09})},\\
{N2}(m,n)
&= \mbox{the number of partitions of $n$ with distinct odd parts 
         with $M_2$-rank $m$}
\\
& \mbox{ (\cite{Be-Ga02}, \cite{Lo-Os09})}.
\end{align*}
Then
\begin{align}
\sum_{n=0}^\infty\sum_{m} N(m,n) z^m q^n &= R(z,q),
\mylabel{eq:NRid}\\
\sum_{n=0}^\infty\sum_{m} \overline{N}(m,n) z^m q^n &= H(z,q),
\mylabel{eq:NBHid}\\
\sum_{n=0}^\infty\sum_{m} N2(m,n) (-1)^n z^m q^n &= K(z,q),
\mylabel{eq:N2Kid}
\end{align}
so that
\begin{align}
R(1,q) &= \prod_{n=1}^\infty \frac{1}{(1-q^n)},
\mylabel{eq:R1}\\
H(1,q) &= \prod_{n=1}^\infty \frac{(1+q^n)}{(1-q^n)}
= \prod_{n=1}^\infty \frac{(1-q^{2n})}{(1-q^n)^2},
\mylabel{eq:H1}\\
K(1,q) &= \prod_{n=1}^\infty \frac{(1-q^{2n-1})}{(1-q^{2n})}
= \prod_{n=1}^\infty \frac{(1-q^{n})}{(1-q^{2n})^2},
\mylabel{eq:K1}\\
K(1,-q) &= \prod_{n=1}^\infty \frac{(1+q^{2n-1})}{(1-q^{2n})}.
\mylabel{eq:K1b}
\end{align}


We now collect our generalizations of \eqn{HR1}--\eqn{HR4} into
\begin{theorem}
\mylabel{thm:rankthm}
\begin{align}
&(zq)_\infty (z^{-1}q)_\infty (q)_\infty 
\sum_{n=0}^\infty \frac{q^{n^2}}{(zq)_n (z^{-1}q)_n} 
=(zq)_\infty (z^{-1}q)_\infty (q)_\infty R(z,q) 
\mylabel{eq:NEWrankid}\\
&= \frac{1}{2}\sum_{n=0}^\infty 
\left( \sum_{j=0}^{[n/2]} (-1)^{n+j}
(z^{n-3j}  + z^{3j-n}) q^{\frac{1}{2}(n^2-3j^2) + \frac{1}{2}(n-j)}\right.
\nonumber\\
& \qquad \qquad + \left. \sum_{j=1}^{[n/2]} (-1)^{n+j}
(z^{n-3j+1}  + z^{3j-n-1}) q^{\frac{1}{2}
(n^2-3j^2) + \frac{1}{2}(n+j)}\right),
\nonumber
\end{align}
\begin{align}
&(1+z)(zq)_\infty (z^{-1}q)_\infty (q)_\infty 
\sum_{n=0}^\infty \frac{(-1)_n q^{\frac{1}{2}n(n+1)}}{(zq)_n (z^{-1}q)_n} 
=(1+z)(zq)_\infty (z^{-1}q)_\infty (q)_\infty H(z,q)
\nonumber\\
&\qquad= \sum_{n=0}^\infty 
 \sum_{\abs{m} \le [n/2]} (-1)^{n+m} (z^{n-2\abs{m}+1}  + z^{2\abs{m}-n}) q^{\frac{1}{2}(n^2-2m^2) + \frac{1}{2}n}
\mylabel{eq:CONJ1a}\\
\nonumber\\
&\qquad= \sum_{n=0}^\infty 
 \sum_{\abs{m} \le [n/3]} (-1)^{n}
(z^{n+1-4\abs{m}}  + z^{4\abs{m}-n}) q^{\frac{1}{2}(n^2-8m^2) + \frac{1}{2}n},
\mylabel{eq:CONJ1b}
\end{align}

\begin{align}
&(zq^2;q^2)_\infty (z^{-1}q^2;q^2)_\infty (q^2;q^2)_\infty 
\sum_{n=0}^\infty \frac{(-1)^n q^{n^2} (q;q^2)_n}{(zq^2;q^2)_n 
(z^{-1}q^2;q^2)_n} 
\mylabel{eq:CONJ2}\\
&=(zq^2;q^2)_\infty (z^{-1}q^2;q^2)_\infty (q^2;q^2)_\infty K(z,q)
\nonumber\\
&= \sum_{n=0}^\infty 
\left( \sum_{m=0}^{n} (-1)^{n}
z^{m-n} q^{\frac{1}{2}(2n^2-m^2) + \frac{1}{2}(2n-m)}
+ \sum_{m=1}^{n} (-1)^{n}
z^{n-m+1} q^{\frac{1}{2}
(2n^2-m^2) + \frac{1}{2}(2n+m)}\right).
\nonumber
\end{align}
\end{theorem}

In view of \eqn{R1}--\eqn{K1} we see that the Hecke-Rogers identities 
\eqn{HR1}--\eqn{HR4} follow by putting $z=1$ in \eqn{NEWrankid}--\eqn{CONJ2}.

\begin{cor}
\mylabel{corol:mock3}
Each of Ramanujan's third order mock theta functions
has a Hecke-Rogers-type double sum representation.
\end{cor}
\begin{remark}
Previously the only such representations were known for
the third order functions $\psi(q)$ (Andrews \cite{An13}), and
$\phi(q)$, $\nu(q)$ (Mortenson \cite{Mo13}).
We also note that Hickerson and Mortenson \cite{Hi-Mo12a}
have found Hecke-Rogers double
sum representations for all the classical mock theta functions except
those of third order. 
\end{remark}
\begin{proof}
We recall the universal mock theta function
\beq
g(x,q) := 
x^{-1}\left(-1 + \sum_{n=0}^\infty \frac{q^{n^2}}{(x)_{n+1} (x^{-1}q)_n}\right).
\mylabel{eq:gdef}
\eeq
It is well-known that each of Ramanujan's third mock theta functions
$f(q)$, $\phi(q)$, $\psi(q)$, $\chi(q)$, $\omega(q)$, $\nu(q)$, $\rho(q)$
can be written solely in terms of $g(x,q)$. These expressions have been
cataloged by Hickerson and Mortenson \cite[(5.4)--(5.10)]{Hi-Mo12}.
The result follows from Theorem \thm{rankthm} since
\beq
g(x,q) = x^{-1}\left(-1 + \frac{1}{1-x} R(x,q) \right).
\mylabel{eq:gR}
\eeq
\end{proof}

A striking example is Ramanujan's third order mock theta function
\beq
f(q) = \sum_{n=0}^\infty \frac{q^{n^2}}{(-q;q)_n^2} = R(-1,q).
\mylabel{eq:fdef}
\eeq
Putting $z=-1$ in \eqn{NEWrankid} gives
\beq
f(q) = \frac{(q)_\infty}{(q^2;q^2)_\infty^2}
\sum_{n=0}^\infty \sum_{m=-[n/2]}^{[n/2]}
 \sgn(m) q^{\frac{1}{2}(n^2-3m^2) + \frac{1}{2}(n-m)},
\mylabel{eq:HRf}
\eeq
where $\sgn(m)=1$ if $m\ge 0$ and otherwise $\sgn(m)=-1$.
We may rewrite this identity as
\beq
\sum_{n=0}^\infty q^{\frac{1}{2}n(n+1)} \, f(q) = 
\sum_{n=0}^\infty \sum_{m=-[n/2]}^{[n/2]}
 \sgn(m) q^{\frac{1}{2}(n^2-3m^2) + \frac{1}{2}(n-m)}.
\mylabel{eq:HRfv2}
\eeq
This gives a recurrence for the coefficients of the $q$-series of $f(q)$.
A similar but different result was found by Imamo{\=g}lu, Raum and Richter 
\cite[Theorem 1.1]{Im-Ra-Ri} using the method of holomorphic projection
applied to harmonic weak Maass forms.

By putting $z=-1$ in \eqn{CONJ2} we obtain a similar
identity for the second order mock theta function
\beq
\mu(q) = \sum_{n=0}^\infty \frac{(-1)^n q^{n^2} (q;q^2)_n}{(-q^2;q^2)_n^2} = K(-1,q),
\mylabel{eq:mudef}
\eeq
namely
\beq
\mu(q) = \frac{(q^2;q^2)_\infty}{(q^4;q^4)_\infty^2}
\sum_{n=0}^\infty \sum_{m=-n}^{n}
 \sgn(m) (-1)^m q^{\frac{1}{2}(2n^2-m^2) + \frac{1}{2}(2n-m)},
\mylabel{eq:HRmu}
\eeq
or
\beq
\sum_{n=0}^\infty q^{n(n+1)} \, \mu(q) = 
\sum_{n=0}^\infty \sum_{m=-n}^{n}
 \sgn(m) (-1)^m q^{\frac{1}{2}(2n^2-m^2) + \frac{1}{2}(2n-m)},
\mylabel{eq:HRmuv2}
\eeq
This confirms an identity found earlier by
Hickerson and Mortenson \cite{Hi-Mo12a}.                              
For completeness we examine \eqn{CONJ1a} near $z=-1$.
We divide both sides by $(1+z)$,
let $z\to-1$ and simplify to obtain
\beq
\sum_{n=0}^\infty q^{\frac{1}{2}n(n+1)}
\sum_{n=1}^\infty \frac{ q^{\frac{1}{2}n(n+1)}}
                       {(-q;q)_n (1+q^n)}
=
\sum_{n=1}^\infty \sum_{m=0}^{[n/2]}
(-1)^m (2n - 4m + 1 - \delta_{m,0}(n+1)) 
q^{\frac{1}{2}(n^2-2m^2) + \frac{1}{2}n},
\mylabel{eq:HRnewv2}
\eeq
where $\delta_{m,0}=1$ if $m=0$ and $\delta_{m,0}=0$ otherwise.

\section{Genesis}
\mylabel{sec:genesis}

We now describe how we were led to our two-variable Hecke-Rogers
identities. It began with a new identity for the Andrews \cite{An08b}
spt-function. Let $\spt(n)$ denote the number of smallest parts in the
partitions of $n$. Then
\begin{align}
\sum_{n=1}^\infty \spt(n) q^n &= 
\sum_{n=1}^\infty (q^n + 2 q^{2n} + 3 q^{3n} + \cdots ) \frac{1}{(q^{n+1};q)_\infty}
\mylabel{eq:Sgen}
\\
&=\sum_{n=1}^\infty \frac{q^n}{(1-q^n) (q^{n};q)_\infty}
\nonumber
\\
&=\frac{1}{(q;q)_\infty}
\sum_{n=1}^\infty \frac{q^n (q;q)_{n-1}}{(1-q^n)}.
\nonumber
\end{align}
Andrews \cite[Theorem 4]{An08b} found that
\begin{align}
\sum_{n=1}^\infty \spt(n) q^n &= \frac{1}{(q;q)_\infty}\left(
\sum_{n=1}^\infty \frac{nq^n}{1-q^n} +
\sum_{n=1}^\infty \frac{(-1)^n q^{\frac{1}{2}n(3n+1)} (1 + q^n)}
                       {(1 - q^n)^2}
\right)
\mylabel{eq:sptID2}
\\
&= \frac{1}{(q;q)_\infty}
\sum_{n=1}^\infty \frac{(-1)^{n-1} q^{\frac{1}{2}n(n+1)} (1 - q^{n^2}) (1 + q^n)}
                       {(1 - q^n)^2},
\nonumber
\end{align}
by letting $z=1$ in \cite[Theorem 2.4]{An-Ga-Li12}. This generating
function identity provides an efficient method for calculating
the spt-coefficients.
We find that
\begin{align*}
\sum_{n=1}^\infty \spt(n) q^n &= 
q+3\,{q}^{2}+5\,{q}^{3}+10\,{q}^{4}+14\,{q}^{5}+26\,{q}^{6}+ \cdots 
\\
& \quad \cdots + 
600656570957882248155746472836274\,{q}^{1000} + \cdots
\end{align*}
We tried multiplying by different powers of $\prod_{n=1}^\infty (1-q^n)$ and
stumbled upon
\begin{align}
\prod_{n=1}^\infty (1-q^n)^3 \sum_{n=1}^\infty \spt(n) q^n &= 
q-4\,{q}^{3}-{q}^{5}+9\,{q}^{6}+{q}^{8}+4\,{q}^{9}-16\,{q}^{10}-4\,{q}^{13}
+ \cdots 
\mylabel{eq:S3conj1}
\\
& \quad \cdots 
-1936\,{q}^{990}-900\,{q}^{995}-49\,{q}^{996}-705\,{q}^{1000} + \cdots,
\nonumber
\end{align}
which suggests something is going on ($705=961-256$ is the first non-square!). 
The final result is given below in Corollary \corol{maincor}.
If we let
$$
\sum_{n=1}^\infty a(n) q^n = \prod_{n=1}^\infty (1-q^n)^3 \sum_{n=1}^\infty \spt(n) q^n,
$$
then we find
$$
\sum_{n=0}^\infty a(5n+2) q^n = 
-25\,{q}^{5}+100\,{q}^{15}+25\,{q}^{25}-225\,{q}^{30}-25\,{q}^{40}-100
\,{q}^{45}+400\,{q}^{50} + \cdots,
$$
and one is led to conjecture that
\beq
a(5n+2) = -25 \, a( n/5),
\mylabel{eq:congs35}
\eeq
for $n\ge0$. This is quite surprising since the generating function
for $\spt(n)$ is not a modular form but a quasi-mock modular form from
the work of Bringmann \cite{Br08}.
A similar behaviour occurs for any prime $\ell\equiv \pm 5\pmod{12}$.
See Corollary \corol{heckecong} below for the general result. 

The idea is to find a $z$-analog of \eqn{S3conj1}.
In \cite{An-Ga-Li12} we found a nice $z$-analog of the generating function \eqn{Sgen}
\begin{align*}
S(z,q) &= \sum_{n=1}^\infty \sum_m N_S(m,n) z^m q^n \\
       &= \sum_{n=1}^\infty \frac{q^n (q^{n+1};q)_\infty }
                                 {(zq^n;q)_\infty (z^{-1}q^n;q)_\infty}.
\end{align*}
We note that
$$
S(1,q) = \sum_{n=1}^\infty \spt(n) q^n.         
$$
From \cite[(2.5)]{An-Ga-Li12}, \cite[(3.23)]{An-Ga-Li13} we have the following
generating function identities.
\begin{align}
S(z,q)  
   &= \frac{1}{(q)_\infty} \left(\sum_{n=-\infty}^\infty \frac{(-1)^{n-1}q^{n(n+1)/2}}{(1-zq^n)(1-z^{-1}q^n)}
   - \sum_{n=-\infty}^\infty \frac{(-1)^{n-1}q^{n(3n+1)/2}}{(1-zq^n)(1-z^{-1}q^n)} \right),
\nonumber\\
  &= \frac{1}{(zq)_\infty}
\sum_{n=1}^\infty \frac{(-1)^{n-1} q^{n(n+1)/2}}{(q)_n (1-z^{-1} q^n)} \,
\left(\frac{z^n-1}{z-1}\right).
\mylabel{eq:Szqid2}
\end{align}
We find that
\begin{align}
& (z;q)_\infty (z^{-1};q)_\infty (q;q)_\infty S(z,q)
\mylabel{eq:S3conj2}\\
&\quad ={\frac { \left( -{z}^{2}+2\,z-1 \right) q}{z}}+{\frac { \left( {z}^{4}
-2\,{z}^{2}+1 \right) {q}^{3}}{{z}^{2}}}+{\frac { \left( {z}^{2}-2\,z+
1 \right) {q}^{5}}{z}}+{\frac { \left( -{z}^{6}+2\,{z}^{3}-1 \right) {
q}^{6}}{{z}^{3}}}
\nonumber\\
&\quad +{\frac { \left( -{z}^{2}+2\,z-1 \right) {q}^{8}}{z}}
+{\frac { \left( -{z}^{4}+2\,{z}^{2}-1 \right) {q}^{9}}{{z}^{2}}}+{
\frac { \left( {z}^{8}-2\,{z}^{4}+1 \right) {q}^{10}}{{z}^{4}}} + \cdots,
\nonumber
\end{align}
and we are clearly on the right track.
We are eventually led to conjecture 
\begin{theorem}
\mylabel{thm:mainthm}
\beq
(z)_\infty (z^{-1})_\infty (q)_\infty S(z,q)
= \sum_{n=0}^\infty \sum_{m=0}^n (1 - z^{\frac{1}{2}(n-m)})^2 
z^{\frac{1}{2}(m-n)} \leg{-4}{n} \leg{12}{m} 
q^{\frac{1}{12}( \frac{3n^2-m^2}{2} - 1)},
\mylabel{eq:NEWSid}
\eeq
where $\leg{\cdot}{\cdot}$ is the Kronecker symbol.
\end{theorem}

In this section we prove Theorem \thm{mainthm} using the theory
of basic hypergeometric functions including the method of
Bailey pairs. We start with equation \eqn{Szqid2}.
We will need the following identity \cite[p.216]{An-Ga-Li13}:
\beq
\frac{(q)_\infty}{(z^{-1}q)_\infty} = 1 + 
\sum_{n=1}^\infty \frac{(-1)^n q^{n(n+1)/2} (1-z^{-1})}
                       {(1-z^{-1}q^n)(q)_n}.
\mylabel{eq:FFWid}
\eeq
As noted in \cite[Theorem 3.5]{An-Ga-Li13} this is related to the
spt-like function due to Fokkink, Fokkink and Wang \cite{Fo-Fo-Wa}.
See also Andrews \cite[p.134]{An08b}.

A pair of sequences $(\alpha_n(a, q), \beta_n(a, q))$ is called a
\textit{Bailey pair} with parameters $(a, q)$ if
\beq
\beta_n(a, q) = \sum_{r=0}^n \frac{\alpha_r(a, q)}{(q; q)_{n-r}(aq; q)_{n+r}}
\mylabel{eq:baileypairdef}
\eeq
for all $n \geq 0$.  We will need
\begin{lemma}[Bailey's Lemma]
Suppose $(\alpha_n(a, q), \beta_n(a, q))$ is a Bailey pair with parameters $(a, q)$.
Then $(\alpha'_n(a, q), \beta'_n(a, q))$ is another Bailey pair with parameters
$(a, q)$, where
\begin{align*}
\alpha_n'(a, q)
&= \frac{(\rho_1; q)_n (\rho_2; q)_n (\frac{aq}{\rho_1\rho_2})^n }
            {(\frac{aq}{\rho_1}; q)_n (\frac{aq}{\rho_2}; q)_n } \alpha_n(a, q), \\
\beta_n'(a, q)
&= \sum_{j=0}^n \frac{(\rho_1; q)_j (\rho_2; q)_j (\frac{aq}{\rho_1\rho_2}; q)_{n-j}
       (\frac{aq}{\rho_1\rho_2})^j }
       {(q; q)_{n-j} (\frac{aq}{\rho_1}; q)_n (\frac{aq}{\rho_2}; q)_n }\beta_j(a,q).
    \end{align*}
\end{lemma}
By letting $\rho_1$, $\rho_2\to\infty$ we obtain
\begin{cor}[Limiting Form of Bailey's Lemma]
Suppose $(\alpha_n(a, q), \beta_n(a, q))$ is a Bailey pair with parameters 
$(a, q)$.
Then $(\alpha'_n(a, q), \beta'_n(a, q))$ is another Bailey pair with parameters
$(a, q)$, where
\begin{align*}
\alpha_n'(a, q)
&= a^n q^{n^2} \alpha_n(a,q) \\
\beta_n'(a,q)
&= \sum_{j=0}^n \frac{a^j q^{j^2} \beta_j(a,q)}
                     {(q)_{n-j} }.
\end{align*}
\end{cor}
By letting $n\to\infty$ and using \eqn{baileypairdef} we obtain
\beq
\sum_{j=0}^\infty a^j q^{j^2} \beta_j
= \frac{1}{(aq;q)_\infty} \sum_{r=0}^\infty a^r q^{r^2} \alpha_r,
\mylabel{eq:baileypairlimsum}
\eeq
for any Bailey pair $(\alpha_n,\beta_n)$ with parameters $(a,q)$.

\begin{prop}
\mylabel{propo:prop1}
The following form a Bailey pair.
\begin{align}
  \mylabel{eq:pair1}
       \alpha_n&= \begin{cases}
         1, & n=0, \\
         q^{n^2}(a^n q^n - a^{n-1} q^{-n}), & n\geq1,
       \end{cases}
       &&
       \beta_n= \frac{q^n}{(q)_n (aq;q)_n}.
\end{align}
\end{prop}
\begin{proof}
In \cite[Eqn.(4.1), p.468]{Sl51} we let $c=q$ and $d\to\infty$ to obtain
\beq
\sum_{r=0}^n \frac{ (1-aq^{2r}) q^{r^2-r}a^r}{(a)_{n+r+1} (q)_{n-r}}
= \frac{1}{(q)_n(a)_n}.
\mylabel{eq:slaterid}
\eeq
We have
\beqs
\sum_{r=1}^n \frac{ (1-aq^{2r}) q^{r^2-r}a^r}{(aq)_{n+r} (q)_{n-r}}
= (1-a)\left( \frac{1}{(q)_n(a)_n} - \frac{(1-a)}{(a)_{n+1} (q)_n}\right),
\eeqs
\beqs
\sum_{r=1}^n \frac{ q^{r^2}(a^r q^r - a^{r-1}q^{-r})}
                       { (aq)_{n+r} (q)_{n-r} }
= (1-a^{-1})\left( \frac{1}{(q)_n(a)_n} - \frac{1}{(aq)_{n} (q)_n}\right),
\eeqs

\begin{align*}
\frac{1}{(aq)_n (q)_n} + 
&\sum_{r=1}^n \frac{ q^{r^2}(a^r q^r - a^{r-1}q^{-r})}
                       { (aq)_{n+r} (q)_{n-r} }
= -a^{-1}\frac{(1-a)}{(q)_n(a)_n} + \frac{a^{-1}}{(aq)_{n} (q)_n}\\
&= a^{-1}\left(\frac{-1}{(a)_n (aq)_{n-1}} + \frac{1}{(aq)_n (q)_n}\right)
= a^{-1}\frac{(-(1-aq^n)+1)}{(a)_n (aq)_{n}} 
= \frac{q^n}{(aq)_n (q)_n},
\end{align*}
and the result follows.
\end{proof}

\begin{prop}
\mylabel{propo:prop2}
For any nonnegative integer $n$ we have
\beq
\sum_{j=0}^n \frac{a^j q^{j^2+j} }{(q)_{n-j} (q)_j (aq)_j}
=
\sum_{j=0}^n \frac{(-1)^j a^j q^{j(j+1)/2}}
                  {(q)_{n-j}(aq)_n}.
\mylabel{eq:A1}
\eeq
\end{prop}
\begin{proof}
We need \cite[p.241, (III.7)]{Ga-Ra-book}:
\beqs
\qbas{q^{-n}}{\beta}{q}{z}{\gamma}
= \frac{(\gamma/\beta)_n}{(\gamma)_n}
\qbasc{q^{-n}}{\beta}{\frac{\beta zq^{-n}}{\gamma}}{q}{q}{\frac{\beta q^{1-n}}{\gamma}}{0}.
\eeqs
We make the substitutions $\gamma = aq$, $\beta=c^{-1}q$, $z=acq^{n+1}$, so
that $\frac{\beta zq^{-n}}{\gamma} = q$ and
\beqs
\qbas{q^{-n}}{c^{-1}q}{q}{acq^{n+1}}{aq}
= \frac{(ac)_n}{(aq)_n}
\qbasc{q^{-n}}{c^{-1}q}{q}{q}{q}{\frac{c^{-1} q^{1-n}}{a}}{0}.
\eeqs
Letting $c\to0^{+}$ and then dividing both sides by $(q)_n$ we obtain
\beqs
\sum_{j=0}^n \frac{(q^{-n};q)_j (-1)^j a^j q^{j(j+3)/2 + nj}}
                  { (q)_n (q)_j (aq)_j}
=
\sum_{j=0}^n \frac{(q^{-n};q)_j a^j q^{j(1+n)}}
                  {(q)_n (aq)_n }.
\eeqs
The result follows by \cite[p.233, (I.12)]{Ga-Ra-book}.
\end{proof}

\begin{prop}
\mylabel{propo:prop3}
\begin{align}
&\sum_{m=0}^\infty \sum_{0 \le k < m/3}
(-1)^{m+k} z^{m-3k} q^{\frac{1}{2}(m^2-3k^2) + \frac{1}{2}(m-k)}
=
\sum_{m=0}^\infty \sum_{m/3 <  k \le  m/2}
(-1)^{m+k} z^{-m+3k} q^{\frac{1}{2}(m^2-3k^2) + \frac{1}{2}(m-k)},
\mylabel{eq:SPHR1}\\
&\sum_{m=0}^\infty \sum_{1 \le k < (m+1)/3}
(-1)^{m+k} z^{m-3k+1} q^{\frac{1}{2}(m^2-3k^2) + \frac{1}{2}(m+k)}
=
\sum_{m=0}^\infty \sum_{(m+1)/3 <  k \le  m/2}
(-1)^{m+k} z^{-m+3k-1} q^{\frac{1}{2}(m^2-3k^2) + \frac{1}{2}(m+k)}.
\mylabel{eq:SPHR2}
\end{align}
\end{prop}
\begin{proof}
We show \eqn{SPHR1} by showing that the coefficient of $z^j$ agree
on both sides for $j\ge1$. On the left side this occurs when $m=3k+j$,
where $k\ge0$. We have
$$
\mbox{Coefficient of $z^j$ in LHS\eqn{SPHR1}} =
\sum_{k \ge 0} (-1)^j q^{3k^2 + 3kj + \frac{1}{2}j(j+1) + k}.
$$
On the right side we need $m=3k-j$, and $2k\le m$ so that $k \ge j$,
and we have
\begin{align*}
\mbox{Coefficient of $z^j$ in RHS\eqn{SPHR1}} &=
\sum_{k \ge j} (-1)^j q^{3k^2 - 3kj + \frac{1}{2}j(j-1) + k} \\
&=
\sum_{k \ge 0} (-1)^j q^{3k^2 + 3kj + \frac{1}{2}j(j+1) + k}\\
&=\mbox{Coefficient of $z^j$ in LHS\eqn{SPHR1}},  
\end{align*}
by replacing $k$ by $k+j$ in the first summation. This proves \eqn{SPHR1}.
The proof of \eqn{SPHR2} is similar.
\end{proof}

This proposition leads to a new version of the Hecke-Rogers 
identity \eqn{HR1}.

\begin{cor}
\mylabel{cor:cor1}
\begin{align*}
\prod_{n=1}^\infty(1-q^n)^2 
&= 
2\sum_{m=0}^\infty \left(\sum_{0 \le k < m/3}
(-1)^{m+k} q^{\frac{1}{2}(m^2-3k^2) + \frac{1}{2}(m-k)} \right. \\
&\qquad \left. + \sum_{1 \le k < (m+1)/3}
(-1)^{m+k} q^{\frac{1}{2}(m^2-3k^2) + \frac{1}{2}(m+k)}\right)
+ \sum_{m=0}^\infty q^{3m^2 +m} - \sum_{m=1}^\infty q^{3m^2-m}.
\end{align*}
\end{cor}
\begin{remark}
A finite form of this result was found earlier by the author and
Alex Berkovich \cite{Be-Ga03}.
\end{remark}
\begin{proof}
From \eqn{HR1} we have
\beqs
\prod_{n=1}^\infty (1-q^n)^2 = 
\sum_{m=0}^\infty \left(\sum_{k=0}^{[m/2]}
(-1)^{m+k} q^{\frac{1}{2}(m^2-3k^2) + \frac{1}{2}(m-k)} 
+ \sum_{k=1}^{[m/2]}
(-1)^{m+k} q^{\frac{1}{2}(m^2-3k^2) + \frac{1}{2}(m+k)} \right).
\eeqs
The result follows by splitting each of these sums using 
\eqn{SPHR1}--\eqn{SPHR2} with $z=1$.
\end{proof}

We are now ready to prove Theorem \thm{mainthm}.
First we write \eqn{NEWSid} in the following equivalent form
\begin{align}
&(z)_\infty (z^{-1})_\infty (q)_\infty S(z,q) 
\mylabel{eq:EQNEWSid}\\
&= \sum_{n=0}^\infty \left( \sum_{j=0}^{[n/3]} (-1)^{n+j}
(z^{n-3j}  -2 + z^{3j-n}) q^{\frac{1}{2}(n^2-3j^2) + \frac{1}{2}(n-j)}\right)
\nonumber\\
& \qquad + \left. \sum_{j=1}^{[n/3]} (-1)^{n+j}
(z^{n-3j+1}  -2 + z^{3j-n-1}) q^{\frac{1}{2}
(n^2-3j^2) + \frac{1}{2}(n+j)}\right).
\nonumber
\end{align}
We will prove \eqn{NEWSid} by showing that the coefficient of $z^k$ on
both sides agree for each $k$. 
Since both sides are symmetric in $z$, $z^{-1}$ we may assume $k\ge0$.
From \eqn{Szqid2} we have
\begin{align}
(z)_\infty (z^{-1})_\infty (q)_\infty S(z,q)
&=
(z^{-1})_\infty (q)_\infty 
\sum_{n=0}^\infty \frac{(-1)^{n-1} q^{n(n+1)/2}(1-z^n)}
                       {(q)_n (1-z^{-1} q^n)}
\nonumber\\
&= - (q)_\infty^2 + 
(z^{-1})_\infty (q)_\infty 
\sum_{n=0}^\infty \frac{(-1)^{n} z^n q^{n(n+1)/2}}
                       {(q)_n (1-z^{-1} q^n)}
\mylabel{eq:Szqid2b},
\end{align}
by \eqn{FFWid}.
We now calculate the coefficient of $z^k$ in the Laurent series of
\beq
F(z,q) = 
(z^{-1})_\infty  
\sum_{n=0}^\infty \frac{(-1)^{n} z^n q^{n(n+1)/2}}
                       {(q)_n (1-z^{-1} q^n)}
\mylabel{eq:Fdef}.
\eeq
By Andrews \cite[Eq.(2.2.6)]{An-book} we have
\beq               
F(z,q) = \sum_{n=0}^\infty \frac{(-1)^n z^{-n} q^{\frac{1}{2}n(n-1)}}
                                { (q)_n }
\sum_{m=0}^\infty \frac{(-1)^{m} z^m q^{m(m+1)/2}}
                       {(q)_m }
\sum_{N=0}^\infty z^{-N} q^{mN}.
\mylabel{eq:Fid}
\eeq         
The coefficient of $z^k$ in \eqn{Fid} arises
when $-n+m-N=k$. So we let $n=j-N$, $m=j+k$ where $j\ge N \ge 0$ 
and we find that
\begin{align}
\left[z^k\right]\,F(z,q) &= 
(-1)^k \sum_{N=0}^\infty \sum_{j=N}^\infty
\frac{(-1)^N q^{j^2 + jk + \frac{1}{2}N(N+1) + Nk + \frac{1}{2}k(k+1)}}
     {(q)_{j+k} (q)_{j-N} }
\nonumber\\
&= 
(-1)^k q^{\frac{1}{2}k(k+1)} \sum_{j=0}^\infty 
\left(\sum_{N=0}^j
\frac{(-1)^N q^{\frac{1}{2}N(N+1) + Nk}}
     {(q)_{j+k} (q)_{j-N} }
\right)
q^{j^2 + jk}.
\mylabel{eq:cofzkF}
\end{align}
We now apply the Limiting Form of Bailey's Lemma to the Bailey pair 
\eqn{pair1} to obtain the following Bailey pair with parameters $(a,q)$:
\begin{align*}
       \alpha_n'&= a^n q^{n^2} \alpha_n \\
        &=\begin{cases}
         1, & n=0, \\
         q^{2n^2}(a^{2n} q^n - a^{2n-1} q^{-n}), & n\geq1,
       \end{cases}
       \\
       \beta_n'
&= \sum_{j=0}^n \frac{a^j q^{j^2} \beta_j}
                     {(q)_{n-j} }
\\
&= \sum_{j=0}^n \frac{a^j q^{j^2+j}}
                     {(q)_{n-j} (q)_j (aq)_j}
\\
&=\sum_{j=0}^n \frac{(-1)^j a^j q^{j(j+1)/2}}
                  {(q)_{n-j}(aq)_n},
\end{align*}
by \eqn{A1}. Now using this Bailey pair in \eqn{baileypairlimsum} 
with $a=q^k$ we have
\begin{equation*}
(q)_k 
\sum_{j=0}^\infty 
\left(\sum_{n=0}^j
\frac{(-1)^n q^{\frac{1}{2}n(n+1) + nk}}
     {(q)_{j+k} (q)_{j-n} }
\right)
q^{j^2 + jk}
= \frac{1}{(q^{k+1};q)_\infty} 
\left(1 + \sum_{r=1}^\infty q^{3r^2+rk}(q^{2rk+r} - q^{2rk-k-r})\right),
\end{equation*}
where we have used  the fact that
\beq
(q^{k+1};q)_j = \frac{(q)_{j+k}}{(q)_k}.
\mylabel{eq:qfact}
\eeq
Thus we have
\beq
\sum_{j=0}^\infty 
\left(\sum_{n=0}^j
\frac{(-1)^n q^{\frac{1}{2}n(n+1) + nk}}
     {(q)_{j+k} (q)_{j-n} }
\right)
q^{j^2 + jk}
= \frac{1}{(q;q)_\infty} 
\left(\sum_{r=0}^\infty q^{3r^2+3rk+r}
 - \sum_{r=1}^\infty q^{3r^2+3rk-r-k}\right).
\mylabel{eq:niceid}
\eeq

We are now ready to show that the coefficient of $z^k$ on both sides
of equation \eqn{EQNEWSid} agree for all $k\ge0$. 
\subsubsection*{Case 1} $k=0$.
By \eqn{Szqid2b}, \eqn{Fdef}, \eqn{cofzkF}, \eqn{niceid} we have
\begin{align*}
[z^0] \, \LHS(\eqn{EQNEWSid}) &= -(q)_\infty^2 + 
(q)_\infty 
\sum_{j=0}^\infty 
\left(\sum_{n=0}^j
\frac{(-1)^n q^{\frac{1}{2}n(n+1)}}
     {(q)_{j} (q)_{j-n} }
\right)
q^{j^2}
\\
&=-(q)_\infty^2 + \sum_{r=0}^\infty q^{3r^2+r}
 - \sum_{r=1}^\infty q^{3r^2-r}\\
&=
-2\sum_{m=0}^\infty \left(\sum_{0 \le k < m/3}
(-1)^{m+k} q^{\frac{1}{2}(m^2-3k^2) + \frac{1}{2}(m-k)} \right. \\
&\qquad \left. + \sum_{1 \le k < (m+1)/3}
(-1)^{m+k} q^{\frac{1}{2}(m^2-3k^2) + \frac{1}{2}(m+k)}\right)
\qquad\mbox{(by Corollary \corol{cor1})}\\
&=[z^0]\,\RHS(\eqn{EQNEWSid}).
\end{align*}

\subsubsection*{Case 2} $k \ge 1$.
By \eqn{Szqid2b}, \eqn{Fdef}, \eqn{cofzkF}, \eqn{niceid} we have
\begin{align*}
[z^k] \, \LHS(\eqn{EQNEWSid})  
&=
(q)_\infty (-1)^k q^{\frac{1}{2}k(k+1)} \sum_{j=0}^\infty 
\left(\sum_{n=0}^j
\frac{(-1)^n q^{\frac{1}{2}n(n+1) + nk}}
     {(q)_{j+k} (q)_{j-n} }
\right)
q^{j^2 + jk}
\\
&=
(-1)^k q^{\frac{1}{2}k(k+1)} \left(              
\sum_{r=0}^\infty q^{3r^2+3rk+r}
 - \sum_{r=1}^\infty q^{3r^2+3rk-r-k}\right)
\\
&=
\sum_{j=0}^\infty (-1)^k q^{3j^2+3jk+j +\frac{1}{2}k(k+1)}
 + \sum_{j=1}^\infty (-1)^{k-1} q^{3j^2+3jk-j + \frac{1}{2}k(k-1)}
\\
&=[z^k]\,\RHS(\eqn{EQNEWSid}).
\end{align*}
This completes the proof of \eqn{NEWSid}.

If we divide both sides of \eqn{NEWSid} by $(1-z)(1-z^{-1})$ and let
$z\to1$ we obtain
\begin{cor}
\mylabel{cor:maincor}
\beq
\prod_{n=1}^\infty (1-q^n)^3 \sum_{n=1}^\infty\spt(n) q^n
= -\sum_{n=0}^\infty \sum_{m=0}^n \left( \frac{n-m}{2} \right)^2
\leg{-4}{n} \leg{12}{m}
q^{\frac{1}{12}( \frac{3n^2-m^2}{2} - 1)}.
\mylabel{eq:NEWSPTid}
\eeq
\end{cor}    
Define the function $\alpha(n)$ by
\beq
\sum_{n=1}^\infty \alpha(n) q^n = 
\prod_{n=1}(1 - q^{12n})^3 \sum_{n=1}^\infty \spt(n) q^{12n+1},
\mylabel{eq:alphadef}
\eeq
so that $\alpha(n)=0$ if $n$ is not a positive integer congruent to
$1\pmod{12}$. We have
\begin{cor}
\mylabel{cor:heckecong}
Suppose $\ell \equiv \pm5 \pmod{12}$ is prime. Then
\begin{align*}
&\alpha(\ell n)+\ell^2 \alpha(n/\ell)
=0,\qquad\mbox{if $\ell\equiv5\pmod{12}$},\\
&\alpha(\ell n)-\ell^2 \alpha(n/\ell)
=0,\qquad\mbox{if $\ell\equiv-5\pmod{12}$}.  
\end{align*}
\end{cor}
\begin{proof}
From \eqn{NEWSPTid} we have
\beq
\sum_{n=1}^\infty \alpha(n) q^n 
= -\sum_{n=0}^\infty \sum_{m=0}^n \left( \frac{n-m}{2} \right)^2
\leg{-4}{n} \leg{12}{m}
q^{\frac{3n^2-m^2}{2}}.
\mylabel{eq:NEWSPTidv2}
\eeq
Suppose $\ell$ is prime and $\ell \equiv \pm5\pmod{12}$. Then
we observe that 
$$
3n^2 - j^2 \equiv 0\pmod{\ell} \qquad\mbox{if and only if}\quad
n\equiv j\equiv 0 \pmod{\ell},
$$
since $3$ is quadratic nonresidue mod $\ell$. Hence
$$
\alpha(\ell n) = 
\leg{-4}{\ell} \leg{12}{\ell} 
\ell^2 \alpha(n/\ell),
$$
which gives the result.
\end{proof}

\section{A two-variable Hecke-Rogers identity for the Dyson rank function}
\mylabel{sec:rank}

In this section we prove \eqn{NEWrankid}.
We use the fact that spt-crank function $S(z,q)$ can be written in terms of
the Dyson rank function.
By \cite[Corollary 2.5]{An-Ga-Li12} we have
\beq
S(z,q) = \frac{-1}{(1-z)(1-z^{-1})}
\left( \frac{(q)_\infty}
            {(zq)_\infty (z^{-1}q)_\infty }
- R(z,q) \right),
\mylabel{eq:SRid}
\eeq
so that
\beq
(zq)_\infty (z^{-1}q)_\infty (q)_\infty R(z,q)
=(z)_\infty (z^{-1})_\infty (q)_\infty S(z,q)
+ (q)_\infty^2.
\mylabel{eq:SRids}
\eeq
By \eqn{EQNEWSid} and Proposition \propo{prop3} we have
\begin{align*}
&(zq)_\infty (z^{-1}q)_\infty (q)_\infty R(z,q)\\
&= \sum_{n=0}^\infty \left( \sum_{j=0}^{[n/3]} (-1)^{n+j}
(z^{n-3j}  -2 + z^{3j-n}) q^{\frac{1}{2}(n^2-3j^2) + \frac{1}{2}(n-j)}\right)
\nonumber\\
& \qquad + \left. \sum_{j=1}^{[n/3]} (-1)^{n+j}
(z^{n-3j+1}  -2 + z^{3j-n-1}) q^{\frac{1}{2}
(n^2-3j^2) + \frac{1}{2}(n+j)}\right) + (q)_\infty^2
\nonumber\\
&= \frac{1}{2}\sum_{n=0}^\infty \left( \sum_{j=0}^{[n/2]} (-1)^{n+j}
(z^{n-3j}  -2 + z^{3j-n}) q^{\frac{1}{2}(n^2-3j^2) + \frac{1}{2}(n-j)}\right)
\nonumber\\
& \qquad + \left. \sum_{j=1}^{[n/2]} (-1)^{n+j}
(z^{n-3j+1}  -2 + z^{3j-n-1}) q^{\frac{1}{2}
(n^2-3j^2) + \frac{1}{2}(n+j)}\right) + (q)_\infty^2
\nonumber\\
&= \frac{1}{2}\sum_{n=0}^\infty \left( \sum_{j=0}^{[n/2]} (-1)^{n+j}
(z^{n-3j}  + z^{3j-n}) q^{\frac{1}{2}(n^2-3j^2) + \frac{1}{2}(n-j)}\right)
\nonumber\\
& \qquad + \left. \sum_{j=1}^{[n/2]} (-1)^{n+j}
(z^{n-3j+1} + z^{3j-n-1}) q^{\frac{1}{2}
(n^2-3j^2) + \frac{1}{2}(n+j)}\right),                          
\nonumber
\end{align*}
by \eqn{HR1}.
This completes the proof of \eqn{NEWrankid}.


\section{Two-variable Hecke-Rogers identities for the overpartition rank function}
\mylabel{sec:conj}
In this section we prove \eqn{CONJ1a} and \eqn{CONJ1b}.
First we  prove \eqn{CONJ1a}. 
The other equation \eqn{CONJ1b} we will follow from a transformation of 
Milne \cite{Mi94}.
We need to prove some $q$-hypergeometric identities.

\begin{prop}
\mylabel{propo:propfJTP}
\beq               
(1+z) (z;q)_n (z^{-1};q)_n =
\sum_{j=-n}^{n+1} (-1)^{j+1} 
\frac{(q)_{2n}}
     {(q)_{n+j}(q)_{n-j+1}} (1-q^{2j-1}) z^j q^{\frac{1}{2}j(j-3) + 1}.
\mylabel{eq:fJTPv1}
\eeq
\end{prop}
\begin{proof}
The proof follows from the well-known finite form of the Jacobi triple
product identity \cite[p.49]{An-book}
\beq               
(z;q)_n (z^{-1}q;q)_n =
\sum_{j=-n}^{n} 
     (-1)^j z^j q^{\frac{1}{2}j(j-1)} \qbin{2n}{n+j}_q,
\mylabel{eq:fJTP}
\eeq
by a lengthy but straightforward calculation.
\end{proof}

The proofs of the following proposition and corollary are similar to the 
proofs of
some identities in Chapter 9 of Andrews and Berndt's Volume I
of Ramanujan's Lost Notebook \cite{An-Be-RLNI}.
\begin{prop}
\mylabel{propo:RAML}
\beq
(q)_\infty (zq;q^2)_\infty 
\sum_{n=0}^\infty \frac{(z;q^2)_n}{(zq;q)_n (q)_n}
q^n
= 1 + 2 \sum_{m=1}^\infty (-1)^m z^{m} q^{m^2}.
\mylabel{eq:RAML1}
\eeq             
\end{prop}
\begin{proof}
In this proof we will need the Rogers-Fine identity \cite[(9.1.1)]{An-Be-RLNI}
\beq
\sum_{n=0}^\infty \frac{(\alpha;q)_n}
                       {(\beta;q)_n} \tau^n
=
\sum_{n=0}^\infty
\frac{(\alpha;q)_n (\alpha \tau q/\beta;q)_n \beta^n \tau^n q^{n^2-n}
(1 - \alpha\tau q^{2n})}
{(\beta;q)_n (\tau;q)_{n+1}}.
\mylabel{eq:ROGFIN}
\eeq
We will prove \eqn{RAML1} with $z$ replaced by $z^2$.
Applying Heine's transformation \cite[(III.1)]{Ga-Ra-book}
with $a=z$, $b=-z$, $c=z^2q$ and $z\mapsto q$ we obtain
\begin{align}
(q)_\infty (z^2q;q^2)_\infty 
\sum_{n=0}^\infty \frac{(z^2;q^2)_n}{(z^2q;q)_n (q)_n} q^n
&=
(q)_\infty (z^2q;q^2)_\infty 
\sum_{n=0}^\infty \frac{(z;q)_n (-z;q)_n}{(z^2q;q)_n (q)_n} q^n
\nonumber\\
&=
(q)_\infty (z^2q;q^2)_\infty 
\frac{ (-z;q)_\infty (zq;q)_\infty}
     {(z^2q;q)_\infty (q;q)_\infty}
     \sum_{n=0}^\infty \frac{(-zq;q)_n}
                     {(zq;q)_n} (-z)^n
\mylabel{eq:Heine1}\\
&=
     \sum_{n=0}^\infty \frac{(-z;q)_{n+1}}
                     {(zq;q)_n} (-z)^n.
\nonumber           
\end{align}
In \eqn{ROGFIN} we let $\alpha=-zq$, $\beta=zq$, and $\tau=-z$ to obtain
\beqs
\sum_{n=0}^\infty \frac{(-zq;q)_n}
                       {(zq;q)_n} (-z)^n
=
\sum_{n=0}^\infty
\frac{(-zq;q)_n (zq;q)_n (-1)^n z^{2n}(1-z^2q^{2n+1})q^{n^2} }
{(zq;q)_n (-z;q)_{n+1}}
\eeqs
so that 
\begin{align}
\sum_{n=0}^\infty \frac{(-z;q)_{n+1}}
                       {(zq;q)_n} (-z)^n
&=
\sum_{n=0}^\infty
(-1)^n z^{2n}(1-z^2q^{2n+1})q^{n^2}
\nonumber\\
&=
1 + 2 \sum_{n=1}^\infty (-1)^n z^{2n} q^{n^2}.
\mylabel{eq:falseT1}
\end{align}
The result follows from \eqn{Heine1} and \eqn{falseT1} by replacing $z^2$ by $z$.
\end{proof}

Proposition \propo{RAML} gives some nice false theta function identities.
\begin{cor}
\begin{align}
\sum_{n=0}^\infty \frac{(-z;q)_{n+1}}
                       {(zq;q)_n} (-z)^n
&=
1 + 2 \sum_{n=1}^\infty (-1)^n z^{2n} q^{n^2}
\mylabel{eq:falseT1a}\\
\sum_{n=0}^\infty \frac{(z;q^2)_{n+1} (q;q^2)_n}{(-zq;q)_{2n+1}}z^n
&= 1 + 2 \sum_{j=1}^\infty (-1)^j z^{j} q^{j^2}.
\mylabel{eq:falseT2}
\end{align}
\end{cor}
\begin{proof}
Equation \eqn{falseT1a} is \eqn{falseT1}. To prove \eqn{falseT2}
we need
the following transformation due to Andrews \cite[p.67]{An66}
\beq
\sum_{n=0}^\infty \frac{(a;q^2)_n (b;q)_{2n}}
                       {(q^2;q^2)_n (c;q)_{2n}} t^n
=
\frac{ (b;q)_\infty (at;q^2)_\infty}
     { (c;q)_\infty (t;q^2)_\infty}
\sum_{n=0}^\infty \frac{(c/b;q)_n (t;q^2)_n}
                       {(q;q)_n (at;q^2)_n} b^n.
\mylabel{eq:An66}
\eeq
We let $b=q$, $c=-zq^2$, $t=z$, and $a=zq^2$ in \eqn{An66} to obtain
\beqs
\sum_{n=0}^\infty \frac{(zq^2;q^2)_n (q;q)_{2n}}
                       {(q^2;q^2)_n (-zq^2;q)_{2n}} z^n
=
\frac{ (q;q)_\infty (z^2q^2;q^2)_\infty}
     { (-zq^2;q)_\infty (z;q^2)_\infty}
\sum_{n=0}^\infty \frac{(-zq;q)_n (z;q^2)_n}
                       {(q;q)_n (z^2q^2;q^2)_n} b^n,
\eeqs
\beqs              
\sum_{n=0}^\infty \frac{(zq^2;q^2)_n (q;q^2)_{n}}
                       {(-zq^2;q)_{2n}} z^n
=
\frac{1+zq}{1-z}\,\frac{ (q;q)_\infty (z^2q^2;q^2)_\infty (zq;q)_\infty}
     { (z^2q^2;q^2)_\infty (zq^2;q^2)_\infty}
\sum_{n=0}^\infty \frac{(-zq;q)_n (z;q^2)_n}
                       {(q;q)_n (z^2q^2;q^2)_n} q^n,
\eeqs
\beq
\sum_{n=0}^\infty \frac{(z;q^2)_{n+1} (q;q^2)_n}{(-zq;q)_{2n+1}}z^n
=(q)_\infty (zq;q^2)_\infty 
\sum_{n=0}^\infty \frac{(z;q^2)_n}{(zq;q)_n (q)_n}
q^n.
\mylabel{eq:falseT2a}
\eeq           
Equation \eqn{falseT2} follows from \eqn{falseT2a} and 
\eqn{RAML1}.
\end{proof}

Now we are ready to complete the proof of \eqn{CONJ1a}. As usual  we prove that 
coefficient of $z^k$ on both sides agrees for each $k$.
Let
\begin{align}
\LHS\eqn{CONJ1a}=L(z) &= \sum_{k=-\infty}^\infty \ell_k(q) z^k,
\mylabel{eq:Ldef}\\
\RHS\eqn{CONJ1a}=R(z) &= \sum_{k=-\infty}^\infty r_k(q) z^k.
\mylabel{eq:Rightdef}
\end{align}                             
We see that
$$
L(z) = z\,L(z^{-1}), \qquad R(z) = z\,R(z^{-1}),
$$
so that
$$
\ell_k(q) = \ell_{1-k}(q), \qquad 
r_k(q) = r_{1-k}(q),                
$$
for all $k$. Therefore we may assume that $k\ge0$.
We let $a=\rho^{-1}$, $b=q$, $c=-1$, $d=zq$, $e=z^{-1}q$ in
\cite[(III.10)]{Ga-Ra-book}:
\beq
\qbasc{\rho^{-1}}{q}{-1}{q}{-\rho q}{zq}{z^{-1}q}
= 
\frac{(q, -\rho q, -q; q)_\infty}
     {(zq, z^{-1}q, -\rho q; q)_\infty}
\qbasc{z}{z^{-1}}{-\rho q}{q}{q}{\rho q}{-q}.
\mylabel{eq:III10}
\eeq
We let $\rho\to0^+$ and multiply both sides by 
$(q)_\infty (zq)_\infty (z^{-1}q)_\infty$ to find that
\beq            
(1+z)(zq)_\infty (z^{-1}q)_\infty (q)_\infty 
\sum_{n=0}^\infty \frac{(-1)_n q^{\frac{1}{2}n(n+1)}}{(zq)_n (z^{-1}q)_n} 
= (1+z) (q^2;q^2)_\infty (q)_\infty 
\sum_{n=0}^\infty \frac{(z)_n (z^{-1})_n}
                       {(q^2;q^2)_n} q^n.
\mylabel{eq:CONJ1s1}
\eeq               
From \eqn{fJTPv1} and \eqn{CONJ1s1} we have
\begin{align*}
[z^k] L(z) 
&= (q)_\infty (q^2;q^2)_\infty 
\sum_{n=k-1}^\infty \frac{(-1)^{k+1} (q;q)_{2n} (1 - q^{2k-1}) 
q^{n + \frac{1}{2}k(k-3)+1}}
{(q)_{n+k} (q)_{n-k+1} (q^2;q^2)_n}
\\
&= (q)_\infty (q^2;q^2)_\infty 
\sum_{n=0}^\infty \frac{(-1)^{k+1} (q;q^2)_{n+k-1} (1 - q^{2k-1}) 
q^{n + \frac{1}{2}k(k-1)}}
{(q)_{n+2k-1} (q)_{n-k+1}}
\\
&= 
(q)_\infty (q^2;q^2)_\infty 
\frac{ (q;q^2)_\infty (q^{2k};q)_\infty}
{(q^{2k-1};q^2)_\infty (q)_\infty}
\sum_{n=0}^\infty \frac{(-1)^{k+1} (q^{2k-1};q^2)_{n} (1 - q^{2k-1}) 
q^{n + \frac{1}{2}k(k-1)}}
{(q^{2k};q)_{n} (q)_{n}}
\\
&= 
(q)_\infty (q^{2k};q^2)_\infty 
\sum_{n=0}^\infty \frac{(-1)^{k+1} (q^{2k-1};q^2)_{n}  
q^{n + \frac{1}{2}k(k-1)}}
{(q^{2k};q)_{n} (q)_{n}}
\\
&= (-1)^{k+1} q^{\frac{1}{2}k(k-1)} 
 +  2 \sum_{m=1}^\infty (-1)^{m+k+1} q^{m^2 + (2k-1)m + \frac{1}{2}k(k-1)}
\\
&\hskip 2in\mbox{(by letting $z=q^{2k-1}$ in \eqn{RAML1})}
\\
&= [z^k] R(z),
\end{align*}
as required. This completes the proof of \eqn{CONJ1a}.

We describe Milne's \cite{Mi94} bijective proof that
\beq
\sum_{n=0}^\infty 
\sum_{m=-[n/2]}^{[n/2]} (-1)^{n+m}  
q^{\frac{1}{2}(n^2-2m^2) + \frac{1}{2}n}
= 
\sum_{n=0}^\infty 
 \sum_{m=-[n/3]}^{[n/3]} (-1)^{n}
 q^{\frac{1}{2}(n^2-8m^2) + \frac{1}{2}n}.
\mylabel{eq:MILid}
\eeq
Let
\begin{align*}
\mathcal{S}_1 &= \{ (m,n)\in\mathbb{Z}\times\mathbb{Z}\,:\, n \ge 2\abs{m}\},
\\
\mathcal{S}_2 &= \{ (m,n)\in\mathbb{Z}\times\mathbb{Z}\,:\, n \ge 3\abs{m}\}.
\end{align*}
Define
$$
T \,:\, \mathcal{S}_1 \longrightarrow \mathcal{S}_2
$$
by
$$
T(m,n) = 
\begin{cases}
(\frac{1}{2}m,n) &\mbox{if $m\ge0$ is even},\\
(n-\frac{3}{2}m +\frac{1}{2},3n -4m + 1) &\mbox{if $m\ge1$ is odd},\\
\end{cases}
$$
and
$$
T(-m,n) = (-m_1, n_1) \qquad\mbox{if $T(m,n) = (m_1,n_1)$}.
$$
Milne proved \eqn{MILid} by showing that $T$ is a bijection that satisfies
$$
Q_2(T(m,n)) = Q_1(m,n),
$$
where
\begin{align*}
Q_1(m,n) &= \frac{1}{2}n^2 - m^2 + \frac{1}{2}n,
\\
Q_2(m,n) &= \frac{1}{2}n^2 - 4m^2 + \frac{1}{2}n.
\end{align*}
The same bijection proves that the right sides of \eqn{CONJ1a}
and \eqn{CONJ1b} are equal, since it is not difficult to show
that the transformation $T$ also satisfies
$$
L_{2,1}(T(m,n)) =
\begin{cases}
L_{1,1}(m,n) &\mbox{if $m$ is even}\\
L_{1,2}(m,n) &\mbox{if $m$ is odd}, 
\end{cases}
$$
$$
L_{2,2}(T(m,n)) =
\begin{cases}
L_{1,2}(m,n) &\mbox{if $m$ is even}\\
L_{1,1}(m,n) &\mbox{if $m$ is odd}, 
\end{cases}
$$
where
\begin{align*}
L_{1,1}(m,n) &= n - 2\abs{m} + 1, \\
L_{1,2}(m,n) &= 2\abs{m} - n, \\
L_{2,1}(m,n) &= n - 4\abs{m} + 1, \\
L_{2,2}(m,n) &= 4\abs{m} - n, \\
\end{align*}
and
$$
S_2(T(m,n)) \equiv S_1(m,n) \pmod{2},
$$
where
\begin{align*}
S_1(m,n) &= m + n,\\
S_2(m,n) &= n.
\end{align*}
This completes the proof of \eqn{CONJ1a} and \eqn{CONJ1b}.

\section{A two-variable Hecke-Rogers identity for the $M2$-rank function}
\mylabel{sec:conj2}
In this section we prove \eqn{CONJ2}. First we need a result similar to
Proposition \propo{RAML}.

\begin{prop}
\mylabel{propo:RAMLA}
\beq
\frac{(zq;q)_\infty}
     {(-q;q)_\infty}
\sum_{n=0}^\infty \frac{(-zq;q)_{2n} (-1)^n z^n q^n}
{(z^2q^2;q^2)_n (q^2;q^2)_n} 
= \sum_{m=0}^\infty (-1)^m z^{m} q^{\frac{1}{2}m(m+1)}.
\mylabel{eq:RAML1A}
\eeq             
\end{prop}
\begin{proof}
We apply Heine's transformation \cite[(III.2)]{Ga-Ra-book}
with $a=-zq^2$, $b=-zq$, $c=z^2q^2$, $q\mapsto q^2$ and $z\mapsto q$ to obtain
\beq
\frac{(zq;q)_\infty}
     {(-q;q)_\infty}
\sum_{n=0}^\infty \frac{(-zq;q)_{2n} (-1)^n z^n q^n}
{(z^2q^2;q^2)_n (q^2;q^2)_n} 
= \sum_{n=0}^\infty \frac{(-zq;q^2)_n (-zq)^n}
                         {(-zq^2;q^2)_n},
\mylabel{eq:RAML1B}
\eeq             
after some simplification. The result \eqn{RAML1A} now follows from
\beq
\sum_{n=0}^\infty \frac{(-zq;q^2)_n (-zq)^n}
                         {(-zq^2;q^2)_n}
= 
\sum_{n=0}^\infty (-1)^n z^{n} q^{\frac{1}{2}n(n+1)},
\eeq
which is Entry 9.3.1 in Ramanujan's Lost Notebook 
\cite[Eq.(9.3.1),p.227]{An-Be-RLNI}.
\end{proof}

Now we are ready to complete the proof of \eqn{CONJ2}.
It is clear that the coefficient of $z^k$ on the left side of \eqn{CONJ2}
equals the coefficient of $z^{-k}$. We see that the same is true for the right
side after we rewrite it as
\begin{equation*}
 \sum_{n=0}^\infty (-1)^n q^{\frac{1}{2}n(n+1)} 
  + \sum_{n=1}^\infty \sum_{m=0}^{n-1}
    (-1)^n (z^{n-m} + z^{m-n}) 
    q^{\frac{1}{2}(2n^2-m^2) + \frac{1}{2}(2n-m)}.
\end{equation*}
Thus we may assume that $k\ge0$.
We let $q\to q^2$, $a=\rho^{-1}$, $b=q$, $c=q^2$, $d=zq^2$, $e=z^{-1}q^2$ in
\cite[(III.10)]{Ga-Ra-book}:
\beq
\qbasc{\rho^{-1}}{q}{q^2}{q^2}{\rho q}{zq^2}{z^{-1}q^2}
= 
\frac{(q, \rho q^3, q; q^2)_\infty}
     {(zq^2, z^{-1}q^2, \rho q; q^2)_\infty}
\qbasc{zq}{z^{-1}q}{\rho q}{q^2}{q}{\rho q^3}{q}.
\mylabel{eq:III10b}
\eeq
We let $\rho\to0^+$ and multiply both sides by 
$(q^2;q^2)_\infty (zq^2;q^2)_\infty (z^{-1}q^2;q^2)_\infty$ to find that
\beq            
(zq^2;q^2)_\infty (z^{-1}q^2;q^2)_\infty (q^2;q^2)_\infty 
\sum_{n=0}^\infty \frac{(-1)^n (q;q^2)_n q^{n^2}}
                       {(zq^2;q^2)_n (z^{-1}q^2;q^2)_n} 
= \frac{(q)_\infty}{(-q)_\infty}
\sum_{n=0}^\infty \frac{(zq;q^2)_n (z^{-1}q;q^2)_n}
                       {(q;q^2)_n (q^2;q^2)_n} q^n.
\mylabel{eq:CONJ2s1}
\eeq               
In \eqn{fJTP} we let $q\to q^2$, $z\to zq$ to obtain
\beq               
(zq;q^2)_n (z^{-1}q;q^2)_n =
\sum_{k=-n}^{n} (-1)^{k} 
     z^k q^{k^2} \qbin{2n}{n+k}_{q^2}.
\mylabel{eq:fJTP2}
\eeq
From \eqn{fJTP2} and \eqn{CONJ2s1} we have
\begin{align*}
[z^k] \LHS\eqn{CONJ2} 
&= 
\frac{(q)_\infty}
        {(-q)_\infty}
\sum_{n=k}^\infty \frac{(-1)^{k} (q^2;q^2)_{2n} q^{n + k^2}}
{(q^2;q^2)_{n+k} (q^2;q^2)_{n-k} (q;q)_{2n}}
\\
&= 
\frac{(q)_\infty}
        {(-q)_\infty}
\sum_{n=0}^\infty \frac{(-1)^{k} (-q;q)_{2n+2k} q^{n + k^2+k}}
{(q^2;q^2)_{n+2k} (q^2;q^2)_{n}}
\\
&= 
\frac{(q^{4k+2};q^2)_\infty}
        {(-q;q)_\infty (-q^{2k+1};q)_\infty}
\sum_{n=0}^\infty \frac{(-1)^{k} (-q^{2k+1};q)_{2n} q^{n + k^2+k}}
{(q^{4k+2};q^2)_{n} (q^2;q^2)_{n}}
\\
&= 
\frac{(q^{2k+1};q)_\infty}
        {(-q;q)_\infty }
\sum_{n=0}^\infty \frac{(-1)^{k} (-q^{2k+1};q)_{2n} q^{n + k^2+k}}
{(q^{4k+2};q^2)_{n} (q^2;q^2)_{n}}
\\
&= 
  \sum_{m=0}^\infty (-1)^{m+k} q^{\frac{1}{2}m(m+1) + (2m+1)k + k^2}
\\
&\hskip 2in\mbox{(by letting $z=q^{2k}$ in \eqn{RAML1A})}
\\
&= [z^k] \RHS\eqn{CONJ2}
\end{align*}
as required. This completes the proof of \eqn{CONJ2}.
\section{Two-variable Hecke-Rogers identities for other spt-crank functions}
\mylabel{sec:HRsptc}

Let $\SB(z,q)$ be the generating function for the spt-crank function for
overpartitions \cite{Ga-JS13}. Then
\begin{align}
\SB(z,q)
&= \sum_{n=1}^\infty 
   \frac{q^n (q^{2n+2};q^2)_\infty}
   {(zq^n;q)_\infty (z^{-1}q^n;q)_\infty}
\mylabel{eq:SBdef1}
\\
&= \sum_{n=1}^\infty\sum_{m}N_{\SB}(m,n)z^mq^n.
\mylabel{eq:NSBdef}
\end{align}
We note that
$$
\SB(1,q) =
\sum_{n=1}^\infty 
   \frac{q^n (-q^{n+1};q)_\infty}
   {(1-q^n)^2 (q^{n+1};q)_\infty}
=
\sum_{n=1}^\infty
\sptBar{}{n} q^n,
$$
where $\sptBar{}{n}$ is the
number of 
smallest parts in the overpartitions of $n$, where we are using the convention
that the smallest part is not overlined.
The spt-crank function for overpartitions can be written in terms of
the rank and crank functions for overpartitions.
\beq
\SB(z,q) = \frac{1}{(1-z)(1-z^{-1})}
\sum_{n=1}^\infty \sum_m \left( \overline{N}(m,n) - \overline{M}(m,n)\right) z^m q^n,
\mylabel{eq:SBid}
\eeq
where
\beq
\sum_{n=0}^\infty\sum_{m} \overline{N}(m,n)z^mq^n
= \sum_{n=0}^\infty\frac{(-1;q)_n q^{n(n+1)/2}}
                        {(zq;q)_n  (z^{-1}q;q)_n},
\mylabel{eq:NBdef}
\eeq
and
\beq
\sum_{n=0}^\infty\sum_{m} \overline{M}(m,n)z^mq^n
= \frac{ (-q;q)_\infty (q;q)_\infty }                       
       {(zq;q)_\infty (z^{-1}q;q)_\infty}.
\mylabel{eq:MBdef}
\eeq
We find the following analog of Theorem \thm{mainthm}.
\begin{theorem}
\mylabel{thm:SBthm}
\begin{align}
&(1+z)(z)_\infty (z^{-1})_\infty (q)_\infty \SB(z,q)
\mylabel{eq:NEWSBid}
\\
&\qquad =
\sum_{n=0}^\infty 
\sum_{m=-[n/2]}^{[n/2]}
(-1)^{m+n} (1 - z^{n - 2\abs{m} + 1}) (1 - z^{n - 2\abs{m}}) z^{2\abs{m}-n}
q^{\frac{1}{2}(n^2 - 2m^2) + \frac{1}{2}n}.
\nonumber
\end{align}
\end{theorem}
\begin{proof}
Equation \eqn{NEWSBid} follows in a straightforward manner from \eqn{HR2},
\eqn{CONJ1a},
\eqn{SBid}, \eqn{NBdef} and \eqn{MBdef}.
\end{proof}
If we divide both sides of \eqn{NEWSBid} by $(1-z)(1-z^{-1})$ and let
$z\to1$ we obtain
\begin{cor}
\mylabel{cor:SBcor}
\beq
\prod_{n=1}^\infty (1-q^n)^3 \sum_{n=1}^\infty\sptBar{}{n} q^n
=
\sum_{n=0}^\infty 
\sum_{m=-[n/2]}^{[n/2]}
(-1)^{m+n+1} 
\bin{n - 2\abs{m} + 1} 
q^{\frac{1}{2}(n^2 - 2m^2) + \frac{1}{2}n}.
\mylabel{eq:SBcorid}
\eeq
\end{cor}    

Let $\STwoB(z,q)$ be the generating function for the spt-crank function for
partitions with distinct odd parts and smallest part even \cite{Ga-JS13}. Then
\begin{align}
\STwoB(z,q)
&= \sum_{n=1}^\infty 
   \frac{q^{2n} (q^{2n+2};q^2)_\infty (-q^{2n+1};q^2)\infty}
   {(zq^{2n};q^2)_\infty (z^{-1}q^{2n};q^2)_\infty}
\mylabel{eq:S2def1}
\\
&= \sum_{n=1}^\infty\sum_{m}N_{\STwoB}(m,n)z^mq^n.
\nonumber
\end{align}
We note that
$$
\STwoB(1,q) =
\sum_{n=1}^\infty 
   \frac{q^{2n} (-q^{2n+1};q^2)_\infty}
   {(1-q^{2n})^2 (q^{2n+2};q^2)_\infty}
=
\sum_{n=1}^\infty
\Mspt{n} q^n,
$$
where $\Mspt{n}$ is the
number of 
smallest parts in the partitions of $n$ without repeated odd parts and
with smallest part even. This function was studied by Ahlgren, Bringmann and
Lovejoy \cite{Ah-Br-Lo11}.
Again we find this spt-crank function be written in terms of
the relevant rank and crank functions.
\beq
\STwoB(z,q) = \frac{1}{(1-z)(1-z^{-1})}
\sum_{n=1}^\infty \sum_m \left( {N2}(m,n) - {M2}(m,n)\right) z^m q^n,
\mylabel{eq:S2id}
\eeq
where
\beq
\sum_{n=0}^\infty\sum_{m} {N2}(m,n)z^mq^n
= \sum_{n=0}^\infty\frac{(-q;q^2)_n q^{n^2}}
                        {(zq^2;q^2)_n  (z^{-1}q^2;q^2)_n},
\mylabel{eq:N2def}
\eeq
and
\beq
\sum_{n=0}^\infty\sum_{m} {M2}(m,n)z^mq^n
= \frac{ (-q;q^2)_\infty (q^2;q^2)_\infty }                       
       {(zq^2;q^2)_\infty (z^{-1}q^2;q^2)_\infty}.
\mylabel{eq:M2def}
\eeq
We find the following analog of Theorem \thm{mainthm}.
\begin{theorem}
\mylabel{thm:S2thm}
\begin{align}
&(z;q^2)_\infty (z^{-1};q^2)_\infty (q^2;q^2)_\infty \STwoB(z,-q)
\mylabel{eq:NEWS2id}
\\
&\qquad =
\sum_{n=0}^\infty 
\sum_{m=0}^n 
(-1)^{n} (1 - z^{n -m})^2 z^{m-n}
q^{\frac{1}{2}(2n^2 - m^2) + \frac{1}{2}(2n-m)}.
\nonumber
\end{align}
\end{theorem}
\begin{proof}
From \eqn{S2id}, \eqn{N2def}, \eqn{M2def} we have
\begin{align}
&(z;q^2)_\infty (z^{-1};q^2)_\infty (q^2;q^2)_\infty \STwoB(z,-q)
\mylabel{eq:NEWS2id2}
\\
&=
(zq^2;q^2)_\infty (z^{-1}q^2;q^2)_\infty (q^2;q^2)_\infty 
\sum_{n=0}^\infty\frac{(-1)^n (q;q^2)_n q^{n^2}}
                      {(zq^2;q^2)_n  (z^{-1}q^2;q^2)_n}
- (q;q)_\infty (q^2;q^2)_\infty.
\nonumber
\end{align}


We note that
\begin{align*}
&\sum_{n=0}^\infty 
\sum_{m=1}^{n} (-1)^{n}
z^{n-m+1} q^{\frac{1}{2}
(2n^2-m^2) + \frac{1}{2}(2n+m)}
\\
&\qquad=
\sum_{n=0}^\infty 
\sum_{m=1}^{n} (-1)^{n}
z^{n-(m-1)} q^{\frac{1}{2}
(2n^2-(m-1)^2) + \frac{1}{2}(2n-(m-1))}
\\
&\qquad=
\sum_{n=0}^\infty 
\sum_{m=0}^{n-1} (-1)^{n}
z^{n-m} q^{\frac{1}{2}
(2n^2-m^2) + \frac{1}{2}(2n-m))}.
\end{align*}
Thus from \eqn{HR4}, \eqn{CONJ2} and \eqn{NEWS2id2} we have
\begin{align}
&(z;q^2)_\infty (z^{-1};q^2)_\infty (q^2;q^2)_\infty \STwoB(z,-q)
\mylabel{eq:NEWS2idcopy}
\\
&\qquad =
\sum_{n=0}^\infty 
\sum_{m=0}^n 
(-1)^{n} (z^{m-n} + z^{n-m} - 2)
q^{\frac{1}{2}(2n^2 - m^2) + \frac{1}{2}(2n-m)}
\nonumber
\\
&\qquad =
\sum_{n=0}^\infty 
\sum_{m=0}^n 
(-1)^{n} (1 - z^{n -m})^2 z^{m-n}
q^{\frac{1}{2}(2n^2 - m^2) + \frac{1}{2}(2n-m)},
\nonumber
\end{align}
which is the result.
\end{proof}
If we divide both sides of \eqn{NEWS2id} by $(1-z)(1-z^{-1})$ and let
$z\to1$ we obtain
\begin{cor}
\mylabel{cor:maincorS2}
\beq
\prod_{n=1}^\infty (1-q^{2n})^3 \sum_{n=1}^\infty (-1)^n \Mspt{n} q^n
= \sum_{n=1}^\infty \sum_{m=0}^n (-1)^{n+1} (n - m)^2                            
q^{\frac{1}{2}(2n^2 - m^2) + \frac{1}{2}(2n-m)}.
\mylabel{eq:NEWM2SPTid}
\eeq
\end{cor}    
Define the function $\beta(n)$ by
\beq
\sum_{n=1}^\infty \beta(n) q^n = 
\prod_{n=1}(1 - q^{16n})^3 \sum_{n=1}^\infty (-1)^n \Mspt{n} q^{8n+1},
\mylabel{eq:betadef}
\eeq
so that $\beta(n)=0$ if $n$ is not a positive integer congruent to
$1\pmod{8}$. We have
\begin{cor}
\mylabel{cor:M2heckecong}
Suppose $\ell \equiv\pm3\pmod{8}$ is prime. Then
\begin{align*}
&\beta(\ell n)+\ell^2 \beta(n/\ell)
=0,\qquad\mbox{if $\ell\equiv3\pmod{8}$},\\
&\beta(\ell n)-\ell^2 \beta(n/\ell)
=0,\qquad\mbox{if $\ell\equiv5\pmod{8}$}.  
\end{align*}
\end{cor}
\begin{proof}
From \eqn{NEWM2SPTid}, \eqn{betadef} we have
\begin{align*}
\sum_{n=1}^\infty \beta(n) q^n 
&= \sum_{n=0}^\infty \sum_{j=0}^n (-1)^{n+1} (n-m)^2 q^{2(2n+1)^2 - (2m+1)^2}
\\
&= -\sum_{n=1}^\infty \sum_{m=1}^n 
\left( \frac{n-m}{2} \right)^2
\leg{-4}{n} \leg{4}{m} 
q^{2n^2 - m^2}.
\end{align*}
Suppose $\ell$ is prime and $\ell \equiv \pm3\pmod{8}$. Then
we observe that 
$$
2n^2 - m^2 \equiv 0\pmod{\ell} \qquad\mbox{if and only if}\quad
n\equiv m\equiv 0 \pmod{\ell},
$$
since $2$ is quadratic nonresidue mod $\ell$. Hence
$$
\beta(\ell n) = \leg{-4}{\ell} \, \ell^2 \beta(n/\ell),
$$
which gives the result.
\end{proof}

\section{Concluding remarks}

There are other two-variable Hecke-Rogers identities in the literature.
Andrews \cite{An-84} proved the following identity
\beq
\prod_{n=1}^\infty \frac{ (1-q^n)^2 }{ (1-zq^n) (1-z^{-1}q^{n-1})}
=
\sum_{n=0}^\infty\sum_{m=-n}^n (-1)^{m+n} z^m
q^{ \frac{1}{2}(n^2-m^2) + \frac{1}{2}(n+m)},
\mylabel{eq:ANDID}
\eeq
where $\abs{q}<1$ and $1< \abs{z} < \abs{q}^{-1}$. Andrews used this identity
to show how elementary $q$-series techniques
could be used to prove identities such as \eqn{HR1}--\eqn{HR4}.
Hickerson and Mortenson \cite{Hi-Mo12} studied the function
\beq
f_{a,b,c}(x,y,q):=\sum_{\substack{\sgn (r)=\sgn(s)}} 
\sgn(r)(-1)^{r+s}x^ry^sq^{a\binom{r}{2}+brs+c\binom{s}{2}}.
\mylabel{eq:fabc-def}            
\eeq
They found a general identity for this function in terms of Apell-Lerch sums
and theta functions. Their formula not only proves the known Hecke-Rogers
identities such as \eqn{HR1}--\eqn{HR4} but also leads to new 
straightforward proofs of many of the classical mock theta function identities,
including a new proof of the mock theta conjectures \cite{An-Ga89}. 
It would interesting to determine whether the methods and results of
Hickerson and Mortenson \cite{Hi-Mo12} can be used to give an alternative
proof of our main result Theorem \thm{rankthm}. It would also be interesting
to see whether the theory of mock Jacobi forms \cite{Br-Ra-Ri}, \cite{Da-Mu-Za},
\cite{Zw} could be developed to derive these results.

Our proof of our rank function result \eqn{NEWrankid} depends on first
proving the spt-crank result \eqn{NEWSid} in Theorem \thm{mainthm}.
The proof of \eqn{NEWSid} utilizes the method of Bailey pairs. It is possible 
to give a direct proof of \eqn{NEWrankid} using Bailey pair technology.
We leave this to the interested reader. We were unable to find a proof
of the other rank-type function results \eqn{CONJ1a}--\eqn{CONJ2}
by the method of Bailey pairs.

Lovejoy \cite{Lo12} has found a number of identities that give certain
$q$-hypergeometric sums in terms of two-variable Hecke-Rogers type series
using the method of Bailey pairs. Lovejoy \cite{Lo14} has also found
families of $q$-hypergeometric mock theta multisums in terms of
Hecke-Rogers type double series.

Mortenson \cite{Mo13b} has utilized Lovejoy's \cite{Lo12} method also
obtain some similar identities. We give three examples.
From results in \cite[Section 4.3]{Mo13b} it can be shown that
\begin{align}
&(q)_\infty (1 + z^{-1}) \sum_{n=0}^\infty
\frac{ (-zq;q^2)_n (-z^{-1}q;q^2)_n q^{2n}}{ (q;q^2)_{n+1} }
\mylabel{eq:MORTID1}
\\
&\qquad =  \sum_{n=0}^\infty \sum_{m=0}^{[n/3]}
(z^{n-3m} + z^{3m-n-1}) q^{(n^2-3m^2) + (2n-m)}(1 - q^{4n-4m+6}).
\nonumber
\end{align}
Dividing both sides by $(1+z^{-1})$ and letting $z\to-1$ yields
\begin{align}
&(q)_\infty \sum_{n=0}^\infty
\frac{ (q;q^2)_n q^{2n}}{ (1-q^{2n+1}) }
\mylabel{eq:MORTID1B}
\\
&\qquad =  \sum_{n=0}^\infty \sum_{m=0}^{[n/3]}
(-1)^{m+n}(2n - 6n + 1) q^{(n^2-3m^2) + (2n-m)}(1 - q^{4n-4m+6}).
\nonumber
\end{align}
Similarly, from results in \cite[Section 4.4]{Mo13b} we find that
\begin{align}
&(-q;q^4)_\infty (-q^3;q^4)_\infty (q^4;q^4)_\infty 
(1 + z^{-1}) \sum_{n=0}^\infty
\frac{ (zq;q^2)_n (z^{-1}q;q^2)_n q^{2n}}{ (q;q)_{2n+1} }
\mylabel{eq:MORTID2}
\\
&\qquad = \sum_{n=0}^\infty \sum_{m=0}^{[n/2]}
(-1)^m(z^{m} + z^{-m-1}) q^{\frac{1}{2}(n^2-2m^2) + \frac{1}{2}(3n-2m)},
\nonumber
\end{align}
and
\begin{align}
&\sum_{n=0}^\infty q^{\frac{1}{2}n(n+1)}                          
\sum_{n=0}^\infty
\frac{ (-q;q^2)_n q^{2n}}{ (-q^2;q^2)_{n} (1 + q^{2n+1}) }
\mylabel{eq:MORTID2B}
\\
&\qquad = \sum_{n=0}^\infty \sum_{m=0}^{[n/2]}
(2m + 1) q^{\frac{1}{2}(n^2-2m^2) + \frac{1}{2}(3n-2m)}.
\nonumber
\end{align}
Also, from results in \cite[Section 4.6]{Mo13b} we find that
\begin{align}
&(q;q^2)_\infty (q;q)_\infty  
(1 + z^{-1}) \sum_{n=0}^\infty
\frac{ (-zq;q)_n (-z^{-1}q;q^2)_n q^{n+1}}{ (q;q^2)_{n} }
\mylabel{eq:MORTID3}
\\
&\qquad = \sum_{n=0}^\infty \sum_{m=0}^{[n/2]}
(-1)^m(z^{n-2m} + z^{-n+2m-1}) q^{\frac{1}{2}(n^2-2m^2) + \frac{1}{2}(3n-2m)},
\nonumber
\end{align}
and
\begin{align}
&\sum_{n=-\infty}^\infty (-1)^n q^{n^2}
\sum_{n=0}^\infty
\frac{ (q;q)_n^2 q^{n}}{ (q;q^2)_{n+1} }
\mylabel{eq:MORTID3B}
\\
&\qquad = \sum_{n=0}^\infty \sum_{m=0}^{[n/2]}
(-1)^{m+n}(2n - 4m + 1) q^{\frac{1}{2}(n^2-2m^2) + \frac{1}{2}(3n-2m)}.
\nonumber
\end{align}

\noindent
\textbf{Acknowledgements}

\noindent
I would like to thank
Kathrin Bringmann, Freeman Dyson, Mike Hirschhorn,
Robert Osburn, Steve Milne, Eric Mortenson and Martin Raum for their comments and 
suggestions. In particular, I thank Steve Milne for earlier pointing out his
bijective proof \cite{Mi94} of \eqn{MILid}, and I thank Eric Mortenson for his 
detailed
comments and the results \eqn{MORTID1}--\eqn{MORTID3B}.                       
Finally, I thank Doron Zeilberger for inviting me to present the preliminary
results of this paper in his Experimental Math Seminar on April 25, 2013.


\bibliographystyle{amsplain}

\end{document}